\documentclass[reqno, a4paper, 11pt]{amsart}
\usepackage[utf8x]{inputenc}
\usepackage[dvips]{epsfig}
\usepackage{amsgen, amstext,amsbsy,amsopn,amssymb, amsthm, 
mathptmx, amsfonts,amssymb,amscd,amsmath,nicefrac,euscript,enumerate,url,verbatim,calc,framed}
\usepackage[all]{xy}

\DeclareMathOperator{\pnt}{\raise 0.5mm \hbox{\large\bf.}}

\newtheorem{thm}{\bf Theorem}[section]
\newtheorem{lemma}[thm]{\bf Lemma}
\newtheorem{cor}[thm]{\bf Corollary}
\newtheorem{pro}[thm]{\bf Proposition}
\newtheorem{conj}[thm]{\bf Conjecture}

\newtheorem{note}[thm]{\bf Note}

\theoremstyle{definition}

\newtheorem{rem}[thm]{\bf Remark}
\newtheorem{ex}[thm]{\bf Example}

\title[Regular sequences of symmetric polynomials]{Regular sequences of power sums and complete symmetric polynomials}

\author{Neeraj~Kumar}
\address{Dipartimento di Matematica, Universit\'{a} di Genova\\
Via Dodecaneso 35, 16146 Genova, Italy}
\email{kumar@dima.unige.it}

\author{Ivan~Martino} 
\address{Matematiska Institutionen, Stockholms Universitet, Stockholm, Sweden}
\email{martino@math.su.se}

\thanks{{\it Date: October $07$, $2011$.}}    
\thanks{
\endgraf
}

\begin{document}

\begin{abstract}
In this article, we carry out the investigation for regular sequences 
of symmetric polynomials in the polynomial ring in three and four variable. 
Any two power sum element in $\mathbb{C}[x_1,x_2,\dots,x_n]$ for $n \geq 3$ 
always form a regular sequence and we state the conjecture when $p_a,p_b,p_c$ 
for given positive integers $a<b<c$ forms a regular sequence in 
$\mathbb{C}[x_1,x_2,x_3,x_4]$. We also provide evidence for this conjecture 
by proving it in special instances. We also prove that any sequence of power sums 
of the form $p_{a}, p_{a+1},\dots, p_{a+ m-1},p_b$ with $m <n-1$ forms a
regular sequence in $\mathbb{C}[x_1,x_2,\dots,x_n]$.  We also provide a partial
evidence in support of conjecture's given by Conca, Krattenthaler and Watanabe 
in \cite{C-K-W paper} on regular sequences of symmetric polynomials.
\end{abstract}

\maketitle

\section{Introduction}
The work in this article is inspired by the work of Conca, Krattenthaler and 
Watanabe on regular sequences of symmetric polynomials \cite{C-K-W paper}.

We introduce some basic definitions, notation and well known results
which we will use in the sequel. Let $S=\mathbb{C}[x_1,\dots,x_n]$ be a polynomial ring. 
We denote by $ p_m(x_1,x_2,\dots,x_n), h_m(x_1,x_2,\dots,x_n)$ and $e_m(x_1,x_2,\dots,x_n)$, the power sum 
symmetric polynomials, complete homogeneous symmetric polynomials and the 
elementary symmetric polynomials of degree $m$ in $S$ respectively, that is:
\[
 \begin{split}
 p_m(x_1,x_2,\dots,x_n):=& \sum_{i=1}^{n}x_i^m, \\
 h_m(x_1,x_2,\dots,x_n):=& \sum_{1 \leq i_{1} \leq i_{2} \leq \cdots \leq i_m \leq n} x_{i_1}x_{i_2}\cdots x_{i_m},\\
 e_m(x_1,x_2,\dots,x_n):=& \sum_{1 \leq i_{1} < i_{2} < \cdots < i_m \leq n} x_{i_1}x_{i_2}\cdots x_{i_m}.
 \end{split}
\]
We will also denote by $p_m(n),\;h_m(n), \text{ and } e_m(n),$ the power sum symmetric polynomials,
 complete homogeneous symmetric polynomials, and the elementary symmetric polynomials 
respectively. When $n$ is clear from the context, we may simply denote them by $p_m,h_m 
\text{ and } e_m$ respectively. 
 
 A sequence of elements $f_1,f_2,\dots, f_k $ in $S$ is a \emph{regular sequence} on $S$ 
if the ideal $(f_1,f_2,\dots,f_k)$ is proper and for each $i$, the image of $f_{i+1}$ is a 
nonzero divisor in $S/(f_1,\dots,f_i)$.
 


We have used \emph{Newton's formulas} for $p_n,h_n \text{ and } e_n$, 
( see equation $2.6^{\prime},2.11^{\prime}$) from Macdonald \cite{I-G-Macdonald}.
These relations together with the Theorem \ref{thm-being-regular} are very helpful in 
investigating regular sequences. 

We have used the \emph{Serre's criterion for normality} (see section $18.3$, 
Theorem $18.15$ \cite{Eisenbud}) for proving primeness
for power sum polynomials in the polynomial ring. Once we know that $p_a,p_b$ 
generates a prime ideal in $\mathbb{C}[x_1,x_2,x_3,x_4]$, we can add
one more polynomial $f$ and conclude that $p_a,p_b,f$ forms a regular sequence for
all $f \notin (p_a,p_b)$. We prove $p_{1},p_{2m}$ generates a prime ideal, where
$m \in \mathbb{N}$, see Proposition \ref{1-2m}. We also prove this in the case of 
consecutive integres $a,a+1$. In fact we prove a more general statement that 
any consecutive power sum $p_{a}, p_{a+1},\dots, p_{a+ m-1}$ 
with $m <n-1$ generates a prime ideal in $\mathbb{C}[x_1,x_2,\dots,x_n]$, see Theorem \ref{n-2 consecutive}.

In general, it turns out to be difficult to find conditions on $\{a,b\}$ such that $(p_a,p_b)$ is a prime ideal. 
We did several computations in CoCoA and found some conditions on $\{a,b\}$ such that 
$(p_a,p_b)$ is a prime ideal, see Conjecture \ref{prime computation}. 
For example when $a$ is prime number, $a \geq5$ and $b=a+m+6d$ with $m \in \{1,5\}$ 
then $(p_a,p_b)$ is a prime ideal.
 
However, a very nice introduction to regular sequences of symmetric polynomials
is given by Conca, Krattenthaler and Watanabe in \cite{C-K-W paper}. 
So, we refer the reader for detailed introduction to \cite{C-K-W paper}. 

\medskip
{ \bf Convention:} By an expression of the form $f_{i_1},f_{i_2},\dots, f_{i_k}$ or
$(f_{i_1},f_{i_2},\dots f_{i_k})$ for the power sum and complete symmetric polynomials, 
we always assume that $i_1<i_2<\cdots <i_k$.
\medskip

\section{Some results on regular sequences}
\medskip

Let us recall some well known results about regular sequences.
\begin{lemma} Let $S=\mathbb{C}[x_1,x_2,\dots,x_n]$ be a polynomial ring. The sequence of 
homogeneous polynomials $f_1, f_2,\dots, f_k$ is a regular sequence in $S$ if and only if 
\[
    H_{S/I}(z)=\frac{\prod_{i=1}^{k}(1-z^{d_i})}{(1-z)^n}.
\]
where $d_i=\deg{f_i}$ and $I=(f_1, f_2,\dots, f_k).$
\end{lemma}

We will use the following characterization very often for proving regular sequence for the 
power sums and complete symmetric polynomials:
\begin{thm}\label{thm-being-regular}
Let $f_i, f_j, f_k \in S=\mathbb{C}[x_1,x_2,x_3]$. The sequence $f_i, f_j, f_k$ is a 
regular sequence if and only if  $f_k\notin (f_i, f_j)$  and 
for any $f$ of degree bigger than $i+j+k$ we have $f\in (f_i, f_j, f_k)$.
\end{thm}
\begin{proof}
If $f_i, f_j, f_k$ is a regular sequence then $f_i$ is  not a zero divisor on $S$, $f_j$ is  
not a zero divisor on $S/(f_i)$ and $f_k$ is not a zero divisor on $S/(f_i,f_j)$. 
This implies $f_k\notin (f_i, f_j)$. 

We know that $(0,0,0)$ is the only solution of the system $(f_i, f_j, f_k)$. 
This means $(0,0,0)$ has multiplicity $i+j+k$ and $(f_i, f_j, f_k)$ 
is the $(i+j+k)$-th power of the maximal ideal. So considering $f$ with degree of $f$ bigger than
$i+j+k$, this implies $f\in (f_i, f_j, f_k)$.
\end{proof}
Of course, there are three possible cases for $f_i, f_j, f_k$ in $\mathbb{C}[x_1,x_2,x_3]$:
\begin{enumerate}
  \item $f_k\notin (f_i, f_j)$, and $f_i, f_j, f_k$ is a regular sequence;
  \item $f_k\notin (f_i, f_j)$, and $f_i, f_j, f_k$ is not a regular sequence,
  \item $f_k\in (f_i, f_j)$, then $f_i, f_j, f_k$ is not a regular sequence.
\end{enumerate}
See an Example \ref{ex-1-4-5}, where $h_k\notin (h_i, h_j)$ and $h_i, h_j, h_k$ is not a regular sequence.
\paragraph{Notation:}
For a subset $A\subset \mathbb{N}^*$, we set
\[
p_A(n)=\{p_a(n):a \in A \} \text{ and } h_A(n)=\{h_a(n):a \in A \}.
\]
\begin{pro} \label{consecutive}
 Let $A\subset \mathbb{N}^*$ be a set of $n$ consecutive elements. Then both $p_A(n)$
and $h_A(n)$ are regular sequences in $k[x_1,\dots,x_n]$.
\end{pro}
\begin{proof}
Refer to Proposition $2.9$ \cite{C-K-W paper} for proof.
\end{proof}

We are going to use the Newton's formulas:
\begin{pro}\label{Newton-formulas}
Let $p_n$ be the power sum symmetric polynomial of 
degree $n$,  $h_n$ be the complete homogeneous symmetric polynomial of degree $n$
and let $e_n$ be the elementary symmetric polynomial of degree $n$. Then
\[
\begin{split}
  ne_n=&\sum_{i=1}^{n} (-1)^{i-1}e_{n-i} p_{i} \text{ for all $n \geq 1.$ }\\
  \text{and   } &\sum_{i=0}^{n} (-1)^{i}e_i h_{n-i} =0 \text{ for all $n \geq 1.$ }
\end{split}
\]
These equations are due to Isaac Newton, see Macdonald \cite{I-G-Macdonald} ( equation $2.6^{\prime},2.11^{\prime}$).
\end{pro}

Next Lemma follows from Eisenstein's Criterion.
\begin{lemma}\label{factor}
Let $R$ be a unique factorization domain and $b \in R$. Suppose that $p$ but not $p^2$
divides $b$ for some irreducible $p \in R$. Then $x^m+b$ is irreducible in $R[x]$ 
\end{lemma}

\section{Symmetric Polynomials in $3$ variables}
\medskip
\subsection{ Power Sums in $3$ variables} 
\medskip
\begin{conj}{\rm (Conca, Krattenthaler, Watanabe)} \\[1mm]
 Let $a,b,c$ be positive integers with $a < b< c$ and $\gcd(a,b,c)=1$. 
Then $p_a,p_b,p_c$ is a regular sequence if and only if $abc \equiv 0 (\mod 6)$. 
\end{conj}
\medskip

\begin{rem}
For this conjecture, the ``only if'' part has been proved in \cite{C-K-W paper}, they provide
partial result in support of the ``if'' part. We have also tried to prove this in some 
special cases, here the only difference is in approach, we provide a nice expression
for $p_c\mod(p_a,p_b)$.
 \end{rem}

\begin{pro}\label{p12n}
Consider the power sum sequence $p_1,p_2,p_n$, then
\[
p_n= \begin{cases}
    3e_3^k\mod (p_1,p_2),  &\text{ if  $ n=3k $;}\\
    0\mod (p_1,p_2),  &\text{ otherwise.}
    \end{cases}
\] 
\end{pro}
\begin{proof}
 As $p_0=3$, we use Newtons formula, see Proposition \ref{Newton-formulas} to write $p_n$.
\[
\begin{split}
 p_1=&e_1=0,\\
p_2=&e_1p_1 - 2 e_2=0 \implies e_2=0, \\
p_3=&3e_3+e_1p_2-e_2p_1=3e_3, \\
p_4=&e_1p_3-e_2p_2+e_3p_1=0, \text{ similarly } p_5=0, \\
p_6=&3e_3^2. \text{ And so on, we continue this way.}
\end{split}
\]
Hence we get $p_n=3e_3^{k} \mod(p_1,p_2)$ if $n=3k$.

\end{proof}

\begin{cor}\label{p12n-cor}
$p_1,p_2,p_n$ is a regular sequence if and only if $ n=3k, k \in \mathbb{N}.$ 
\end{cor}
\begin{proof}
We only need to verify the cases of the form $p_1,p_2,p_n$ where $n=3k, k \in \mathbb{N}$.
Choose any $m>1+2+3k=3(k+1)$, we observe that $p_m \in (p_1,p_2,p_n)$. Hence $p_1,p_2,p_n$ 
is a regular sequence for $ n=3k, k \in \mathbb{N}.$
\end{proof}

\begin{pro}
Consider the sequence $p_1,p_3,p_n$, then
\[
p_n= \begin{cases}
    (-1)^ke_2^k\mod (p_1,p_3),  &\text{ if  $ n=2k $;}\\
    0\mod (p_1,p_3),  &\text{ otherwise.}
    \end{cases}
\] 
\end{pro}
\begin{proof}
Similar to Proposition \ref{p12n}. 
\end{proof}

\begin{cor}\label{p13n-cor}
$p_1,p_3,p_n$ is a regular sequence if and only if $ n=2k, k \in \mathbb{N}.$ 
\end{cor}
\begin{proof}
Similar to Corollary \ref{p12n-cor}.
\end{proof}

\begin{rem}
$p_2,p_3,p_n$ is a regular sequence
for all $n$, see Theorem $2.11$ \cite{C-K-W paper}. 
In the paper \cite{C-K-W paper}, they have given a complete proof.
We present here the slightly tricky argument from their paper, 
they managed to reduce the problem and concluded that it is enough 
to prove this for the case $n=4$. They did computer experiments to 
show this for $n=4$ case. But it follows directly from Proposition $2.9$ 
\cite{C-K-W paper} as $2,3,4$ are consecutive integers. 
\end{rem}

\subsection{Complete symmetric polynomials in $3$ variables} 
\medskip

\begin{conj}\label{h-three}{\rm (Conca, Krattenthaler, Watanabe)} \\[1mm]
Let $A=\{a,b,c\}$ with $a < b< c$. 
Then $h_a,h_b,h_c$ is a regular sequence if and only if the 
following conditions are satisfied:
\begin{enumerate}
 \item [1.] $abc \equiv 0 (\mod 6)$.
 \item [2.] $\text{gcd}(a+1,b+1,c+1)=1$.
 \item [3.] For all $ t \in \mathbb{N}$ with $t> 2$ there exist $d \in A$ 
such that $d+2 \not\equiv 0,1 (\mod t)$.
\end{enumerate}
\end{conj}

\medskip

\begin{rem}
For this conjecture, the ``only if'' part has been proved by authors,
the ``if' part is still open. We are able to give partial proof of this conjecture
under some special choice of $a,b$ and for any $c$, both the ''if`` and the ''only if`` part.
 \end{rem}

\begin{pro}\label{h12n}
Consider the sequence $h_1,h_2,h_n$, then
 \[
h_n= \begin{cases}
    -e_3^k\mod (h_1,h_2),  &\text{ if  $ n=3k $;}\\
    0\mod (h_1,h_2),  &\text{ otherwise.}
    \end{cases}
\]
\end{pro}

\begin{proof}
We know by the Proposition \ref{Newton-formulas} that 
\[
 h_n=e_1h_{n-1}-e_2h_{n-2}+\cdots +(-1)^{n}e_nh_{0}. 
\]
Now as in our case $n=3,$ So $e_n=0 \text{ for } n >4.$ Hence
\[
 \begin{split}
 h_0=&1, \\
 h_1=&e_1h_0=e_1 =0  \\
 h_2=&e_1h_1-e_2h_0=0  \text{ which means } e_2=0, \\
 h_3=&e_1h_2-e_2h_1+e_3h_0=-e_3 \mod (h_1,h_2).\\
\end{split}
\]
In this way, we carry out the simplification for $h_n$, $n \geq 4$ and we arrive at the 
following expression:
\[
h_n= \begin{cases}
    -e_3^k\mod (h_1,h_2),  &\text{ if  $ n=3k $;}\\
    0\mod (h_1,h_2),  &\text{ otherwise.}
    \end{cases}
\]
\end{proof} 

\begin{cor}\label{h12nc}
$h_1,h_2,h_{n}$ is a regular sequence if and only if $n =3k, k \in \mathbb{N}.$ 
\end{cor}
\begin{proof}
Clearly the cases 
$n=3k+1$ and $n= 3k+2$ is ruled out. Now for the case $n=3k$. Let us choose 
$m>1+2+3k=3(k+1)$, clearly $h_m \in (h_1,h_2,h_{3k})$. Hence $ (h_1,h_2,h_{n})$ 
is a regular sequence for $n=3k, k \in \mathbb{N}.$
\end{proof}

\begin{pro}
Consider the sequence $h_1,h_3,h_n$, then
\[
h_n= \begin{cases}
    (-1)^{\frac{n}{2}-1}e_2^{\frac{n}{2}} \mod (h_1,h_3),  &\text{ if  $ n=2k $;}\\
    0\mod (h_1,h_3),  &\text{ if $n=2k+1.$}
    \end{cases}
\] 
\end{pro}

\begin{proof}
Similar to Proposition \ref{h12n}.
\end{proof}

\begin{cor}
$h_1,h_3,h_{n}$ is a regular sequence if and only if $n=2k, k \in \mathbb{N}.$ 
\end{cor}

\begin{proof}
Similar to Corollary \ref{h12nc}.
\end{proof}

\begin{pro}\label{pro-1-4}
Consider the sequence $h_1,h_4,h_n$, then
\[
h_n= \begin{cases}
    e_3^k \mod (h_1,h_4),  &\text{ if  $ n=3k $;}\\
    0 \mod (h_1,h_4),  &\text{ if  $ n=3k+1 $;}\\
    -(k+1)e_2e_3^{k}\mod (h_1,h_4),  &\text{ if $n=3k+2.$}
    \end{cases}
\] 
\end{pro}

\begin{proof}
 Similar to Proposition \ref{h12n}.
\end{proof}

\begin{cor}\label{cor-1-4}
The sequence $h_1,h_4,h_{n}$ is a regular sequence if and only if $n=3k,k\in \mathbb{N}.$
\end{cor}
\begin{proof}
Similar to Corollary \ref{h12nc}.
\end{proof}

\begin{pro}
Consider the sequence $h_2,h_3,h_n$, then
\[
h_n= \begin{cases}
    e_1^{2k-2}e_2^{k+1} \mod (h_2,h_3),  &\text{ if  $ n=4k $;}\\
    e_1^{2k-1}e_2^{k+1} \mod (h_2,h_3),  &\text{ if  $ n=4k+1 $;}\\
    0\mod (h_1,h_2),  &\text{ if $n=4k+2,4k+3.$}
    \end{cases}
\] 
\end{pro}
\begin{proof}
Similar to Proposition \ref{h12n}.
\end{proof}

\begin{cor}\label{cor-2-3}
The sequence $h_2,h_3,h_{n}$ is a regular sequence if and only if  $n=4k, 4k+1$, where $k \in \mathbb{N}.$
\end{cor}

\begin{proof}
Clearly $n=4k+2,4k+3$ is ruled out. 
Now let $m_1>2+3+4k=4(k+1)+1$ and $m_2>2+3+4k+1=4(k+1)+2$ 
then $h_{m_1} \in (h_2,h_3,h_{4k})$ and $h_{m_2} \in (h_2,h_3,h_{4k+1})$. 
Hence $(h_2,h_3,h_{n})$ for all $n=4k,4k+1$, $k \in \mathbb{N}$
is a regular sequence. 
\end{proof}

\section{Symmetric Polynomials in $4$ variables}
\medskip
\subsection {Power sums in $4$ variables}
\medskip
\begin{thm}
Let $p_i$ be the power sum symmetric polynomials of degree $i$ in 
the polynomial ring $S=\mathbb{C}[x_1,x_2,\dots,x_n]$. Let $n \geq 3$,
then $p_a,p_b$ is a regular sequence. 
\end{thm}
\begin{proof}
We know $p_a(n)$ is reducible for $n=1$ and $p_2(n)$ is reducible for $n=2$. 
For $n\geq3$, we will show $p_a(n)$ is an irreducible element. We prove this by 
induction on $n$. For $n=3$, we can write $p_a(3)=x_3^a+g$, where $g=x_1^a+x_2^a \in 
\mathbb{C}[x_1,x_2]$, $g$ is a homogeneous and monic polynomial in both variable, of degree $a$. 
So proving factorization of $g(x_1,x_2)$ is same as proving factorization of $g(x_1,1)$.
Since $g(x_1,1)$ has simple roots, $g$ is a product of a linear forms. Thus Lemma \ref{factor} 
shows that $p_a(3)$ is irreducible. Then, if $n >3$, $p_a(n)=x_n^a +p_a(n-1)$ is irreducible by 
Lemma \ref{factor} by induction.
Therefore the ideal generated by $p_a$ is a prime ideal. Hence $S/(p_a)$ is a domain. 
Now $p_b$ being an irreducible element in $S$, can not be factored into lower degree power sum
polynomials $p_a$. So $p_b$ is a non zero divisor on $S/(p_a)$ for $b > a$. 
Hence $p_a,p_b$ is a regular sequence. 
\end{proof}

\begin{note}
If the characteristic of base field $\mathbb{K}$ is not zero, then above result does not hold. 
Consider the field with $\operatorname{char}(\mathbb{K})=2$, then one has $p_4=p_{2}^2$.
\end{note}

In particular, 
\begin{pro}\label{two p sum}
Let $p_i$ be the power sum symmetric polynomials of degree $i$ in 
the polynomial ring $S=\mathbb{C}[x_1,x_2,x_3,x_4]$. Then $p_a,p_b$ is a regular sequence.
\end{pro}

\medskip

We know that a subset of a regular sequence is a regular sequence. So by Proposition
\ref{consecutive}, $p_{a}, p_{a+1},\dots, p_{a+ m-1}$ is a regular sequence. Let $R={S}/{I}$,
where $S=\mathbb{C}[x_1,x_2,\dots,x_n]$, $I=\langle p_{a}, p_{a+1},\dots, p_{a+ m-1}\rangle$ with $m < n-1$.
Hence $R$ is Cohen Macaulay.
Now we are going to use the Serre Criterion (see section $18.3$, Theorem $18.15$ \cite{Eisenbud})
for proving $m$ consecutive power sum polynomials generates a prime ideal in the polynomial ring $S$.
 Once we know that $I$ is a prime ideal in $S$, we can add one more power sum 
element $p_c$ and conclude that $p_{a}, p_{a+1},\dots, p_{a+ m-1},p_c$ forms a regular 
sequence provided $p_c \notin I$.

\begin{thm}\label{n-2 consecutive}
Let $p_i$ be the power sum symmetric polynomials of degree $i$ in 
the polynomial ring $S=\mathbb{C}[x_1,x_2,\dots,x_n]$, with $n \geq 4$.\\
Then $I=\langle p_{a}, p_{a+1},\dots, p_{a+ m-1} \rangle$ with $m < {n-1}$ is a prime ideal in $S$.
 In particular, $p_{a}, p_{a+1},\dots, p_{a+ m-1},p_{c}$ forms a regular sequence provided
$p_c \notin I.$
\end{thm}
\begin{proof}
Consider $S=\mathbb{C}[x_1,x_2,\dots,x_n]$, $I=(p_{a}, p_{a+1},\dots, p_{a+ m-1})$ with $m < n-1$
and $R={S}/{I}$. Now let us compute the Jacobian of $I$, say Jacobian$:=J$.
\[ J=c
\begin{pmatrix}  x_1^{a-1} &  x_2^{a-1} & \cdots & x_n^{a-1}\\  
                 x_1^{a} &  x_2^{a} & \cdots &  x_n^{a}\\
                 \vdots &  \vdots & \cdots & \vdots \\ 
                 x_1^{a+m-2} &  x_2^{a+m-2} & \cdots & x_n^{a+m-2} \\
\end{pmatrix}
\]
We have taken the coefficients out from each row, where $c=\prod_{i=0}^{m-1}(a+i)$. 
We can ignore the coefficients since we are in the field of characteristic zero 
and $c$ is a unit in $\mathbb{C}$. Let $J^{\prime}= I_m(J)$, denote's the ideal 
generated by $m \times m$ minors of Jacobian. Also $m=\text{ht}(I)$, since $I$ is 
generated by a regular sequence of length $m$. The $m \times m$ submatrices of the 
jacobian are standard Van der monde matrices, we know their determinants. 
So we can write 
\[
J^{\prime}=\langle \; x_{i_1}^{j_1}x_{i_2}^{j_2} \cdots x_{i_m}^{j_m} \prod _{1 \leq a<b \leq m} 
(x_{i_a}-x_{i_b}) \;\rangle \text{ for } 1 \leq i_1 <i_2<\cdots <i_m \leq n, 
\]
and for some positive integers $j_1,j_2,\dots,j_m$. Therefore 
\[ 
I+J^{\prime}=\langle p_{a}, p_{a+1},\dots, p_{a+ m-1},\; x_{i_1}^{j_1}x_{i_2}^{j_2} \cdots 
x_{i_m}^{j_m} \prod _{1 \leq a<b \leq m} (x_{i_a}-x_{i_b}) \rangle.
\]
Claim: $\sqrt{I+J^{\prime}}=( x_1,x_2,\dots,x_n).$ \\
Suppose not, that is, there exists $w \in \mathbb{P}^{n-1}$ with $w \in Z(I+J^{\prime})$.
Then the vector $w$ can have at the most $m-1$ distinct non zero coordinates.
If $w$ has $m$ or more than $m$ distinct non zero coordinates, then $w \notin Z(J^{\prime})$. 
Say $w$ has $v$ distinct non zero coordinates. We can write 
\[
w=(w_1,\dots w_1, w_2,\dots,w_2, \dots, w_v,\dots,w_v,0,0,\dots,0),
\]
where $w_i$ appears $\beta_i$ times and $v\leq {m-1}$. Also $w$ should satisfy $p_{a+i}$ for 
$i=0,1,\dots,m-1$ i.e.
\[
\beta_1 w_1^{a+i} +\beta_2 w_2^{a+i}+ \cdots + \beta_v w_v^{a+i}=0 \text{ for $i=1,2,\dots,m.$ }
\]
This is a system of equation, which can be represented in the matrix form with $m$ rows, $v$ column.
\[
 \begin{pmatrix}  1 &  1 & \cdots & 1\\  
                 w_1 & w_2 & \cdots & w_v\\
                 \vdots &  \vdots & \cdots & \vdots \\ 
                 w_1^{m-1} &  w_2^{m-1} & \cdots & w_v^{m-1} 
\end{pmatrix}
\begin{pmatrix}  \beta_1 w_1^{a+i}\\  
                 \beta_2 w_2^{a+i}\\
                 \vdots  \\ 
                 \beta_v w_v^{a+i}
\end{pmatrix}
=
\begin{pmatrix}  0 \\  
                 0 \\
                 \vdots  \\ 
                 0
\end{pmatrix}
\]
We know that neither $\beta_i=0$ nor $w_i=0$ for $i=1,\dots,v$. So, $\beta_iw_i^{a+i}\neq0$
for $i=1,\dots,v$. We can choose the matrix say $M$ with first $v$ rows out of $m$ rows and look for the 
solution. The matrix $M$ is of full rank since $w_i \neq w_j$ for $i\neq j$,
so the only possible solution has to be the trivial solution. 

Therefore such a $w$ does not exist and hence the claim is proved. 
This implies $\text{ht}(I+J^{\prime})=n$ and $\dim \frac{S}{I+J^{\prime}}=0$. 
The co-dimension of $J^{\prime}$ in $S$ is $n-2$. Hence by Theorem $18.15$ in \cite{Eisenbud},  
$R$ is a product of normal domain. So, we can write $R=R_1\times \cdots \times R_k$. 
Since $R$ is a standard graded $\mathbb{C}$-algebra with $R_0=\mathbb{C}$, 
also $R_0=(R_1)_0\times \cdots \times (R_k)_0=\mathbb{C}^k$. Hence $k=1$. 
Therefore $R$ is a normal domain and $I$ is a prime ideal in $S$.
\end{proof}
In particular,
\begin{pro}
Let $p_i$ be the power sum symmetric polynomials of degree $i$ in 
the polynomial ring $S=\mathbb{C}[x_1,x_2,x_3,x_4]$; 
then $I=(p_i,p_{i+1})$ is a prime ideal in $S$. In particular,  
$p_i,p_{i+1},p_n$ is a regular sequence for all $p_n \notin (p_i,p_{i+1})$.
\end{pro}

\medskip

\begin{lemma}\label{4-roots of unity}
Let $n \geq 5$ be any natural number. If the sum of four distinct $n$-th roots of unity is zero, 
then they must be two pair of opposite sign.

Furthermore, if $n$ is odd number, then sum of four distinct $n$-th roots of unity is never zero.
\end{lemma}
\begin{proof}
We have $z^n=1$. Let us pick four distinct $n$-th roots of unity, call it $z_j$ for 
$j=1,2,3,4$. Each $z_j=x_j+iy_j$, where $x_j,y_j \in \mathbb{R}$ and $|z_j|=1$.\\ 
Claim: If $\sum_{j=1}^{4} z_j=0$, then $z_1,\;z_2,\;z_3,\;z_4$ must be of the form
$z_1,\;z_2,\;-z_1,\;-z_2.$ \\
Let $\sum_{j=1}^{4} z_j=0$ i.e. $\sum_{j=1}^{4} x_j=0$ and $\sum_{j=1}^{4} y_j=0$. 
So, we can write,
\[
x_1+x_2=-(x_3+x_4), \text{ and } y_1+y_2=-(y_3+y_4). 
\]
Now, squaring and adding both the equation, we obtain $2+2(x_1x_2+y_1y_2)=2+2(x_3x_4+y_3y_4)$.
Therefore, we get,
\[
|z_1-z_2|^{2}=2-2(x_1x_2+y_1y_2)=2-2(x_3x_4+y_3y_4)=|z_3-z_4|^{2}. 
\]
So $|z_1-z_2|=|z_3-z_4|$. 
Similarly, we get $|z_2-z_3|=|z_1-z_4|$. So, four distinct $z_j$ form a parellelogram. 
The diagonals of a parellelogram intersect at mid point. Hence solving $\frac{z_1+z_3}{2}=\frac{z_2+z_4}{2}$
and $\sum_{j=1}^{4} z_j=0$, we conclude that $z_3=-z_1$ and $z_4=-z_2$. So, we obtain 
four distinct roots of unity as $z_1,\;z_2,\;-z_1,\;-z_2$. 

Furthermore, $z_1$ and $-z_1$ both can not be $n$-th roots of unity for any $n$ odd number.
\end{proof}

\begin{pro}\label{1-2m}
Let $I=(p_{1},p_{2m})$, where $m \in \mathbb{N}$. Then $I$ is a prime ideal 
in $\mathbb{C}[x_1,x_2,x_3,x_4]$. Therefore $p_{1},p_{2m},p_n$ form a regular 
sequence for all $p_n \notin (p_{1},p_{2m})$. 
\end{pro}
\begin{proof}
For $m=1$, it follows from Proposition \ref{two p sum}. Let $m>1$.
Consider $S=\mathbb{C}[x_1,x_2,x_3,x_4],\;I=(p_{1}, p_{2m})$, we know by Theorem
 \ref{n-2 consecutive} that $p_1,\;p_{2m}$ is a regular sequence in $S$. So $\text{ht}(I)=2$.
Let $R={S}/{I}$. Now let us compute the Jacobian of $I$, say Jacobian$:=J$.
\[ J=(2m-1)
\begin{pmatrix}  1 &  1 & 1 & 1\\  
                 x_1^{2m-1} &  x_2^{2m-1} & x_3^{2m-1} &  x_4^{2m-1}\\
 \end{pmatrix}
\]
We can ignore the coefficients $2m-1$, since we are in the field of 
characteristic zero and $2m-1$ is a unit in $\mathbb{C}$. 
Let $J^{\prime}= I_2(J)$, denote's the ideal generated by 
$2 \times 2$ minors of $J$. 
So, we can write 
$J^{\prime}=\langle \; x_{j}^{2m-1}-x_{i}^{2m-1} \; \rangle$ for $ 1 \leq i \leq j \leq 4$. 
Therefore consider
\[ 
I+J^{\prime}=\langle \; p_{1}, p_{2m},\; x_{j}^{2m-1}-x_{i}^{2m-1} \; \rangle \text{ for $ 1 \leq i \leq j \leq 4$. }
\]
Claim: $\sqrt{I+J^{\prime}}=(x_1,x_2,x_3,x_4).$ \\
Suppose not, i.e. there exists $w \in \mathbb{P}^{3}$ with $w \in Z(I+J)$.
Let $w=(w_1,w_2,w_3,w_4)$. If one of $w_i$ is zero, then it is easy to see 
all the $w_i$'s are zero. So we assume none of $w_i$ is zero. Also assume $w_i \neq w_j$
for $i \neq j$. Since $w$ is in $\mathbb{P}^{3}$, we can make $w_1=1$ if $w_1 \neq 0$.
So let $w=(1,x,y,z)$. Since $w \in Z(I+J)$ implies $1=x^{2m-1}=y^{2m-1}=z^{2m-1}$ and
$x^{2m-1}=y^{2m-1}=z^{2m-1}$. Also $w$ satisfies $p_{1}, p_{2m}$. Therefore 
\[
 1+x+y+z=0 \text{ and } 1+x^{2m}+y^{2m}+z^{2m}=0.
\]
Both the equation reduces to existence of solution of $1+x+y+z=0$. We assumed all
the coordinates are distinct, We use the fact that all the $x$ or $y$ or $z$ is $(2m-1)$-th
roots of unity, say $1,\zeta_1,\dots,\zeta_{2m-2}$. 
Now, it follows from the Lemma \ref{4-roots of unity} that $1+\zeta_i+\zeta_j+ \zeta_k \neq 0$ for distinct $i,j,k$.
For $\zeta_i$'s to be distinct, one must have $m>2$. But for $m=2$, one has cube
roots of unity, so one of $\zeta_i=\zeta_j$ for some $i,j$. In that case it is clear that
there is no solution. Now it is easy to verify that if $w=(1,x,y,y)$ or $w=(1,x,x,x)$, then 
also, there does not exist solution of $p_1(w)=0$. So, the only possible solution has to 
be the trivial solution. Hence the claim is proved. 
This implies $\text{ht}(I+J^{\prime})=4$ and $\dim \frac{S}{I+J^{\prime}}=0$. 
The co-dimension of $J^{\prime}$ in $S$ is $2$. 
Hence by Theorem $18.15$ in \cite{Eisenbud},  
$R$ is a product of normal domain. So, we can write $R=R_1\times \cdots \times R_k$. 
Since $R$ is a standard graded $\mathbb{C}$-algebra with $R_0=\mathbb{C}$, 
also $R_0=(R_1)_0\times \cdots \times (R_k)_0=\mathbb{C}^k$. Hence $k=1$. 
Therefore $R$ is a normal domain and $I$ is a prime ideal in $S$.
\end{proof}

\medskip
Computer calculations using CoCoA suggest the following conjecture:
\begin{conj}
Let $p_i$ be the power sum symmetric polynomial of degree $i$ in 
the polynomial ring $S=K[x_1,x_2,x_3,x_4]$. Let $A=\{a,b,n\}$ with $a<b<n$,
then $p_A(4)$ is a regular sequence if and only if $A$ satifies the 
following conditions:
\begin{itemize}
 \item [1.] If $a$ is odd and $b$ is even, then for any $n.$
 \item [2.] If $a$ is odd and $b$ is odd, then for any $n$ even. 
 \item [3.] If $a$ is even, say $a=2m$, with $m$ odd, then for all $n$ 
           provided $\lambda \neq 4k,k\in \mathbb{N} $ where $\lambda=b-a$ and If 
           $\lambda =4k,$ then for all $n$ of the form $4l+2$, with $l \in \mathbb{N}.$
                   
\item [4.] If $a$ is even, say $a=2m$, with $m$ even, then for all $n$ provided 
          $b \neq 3a \text{ and } n \neq (2k+1)a,k \in \mathbb{N}.$
\item [5.] $(a,b,n)$ should not be of the form $(a,2a,5a)$, irrespective of $a$, being even or odd.
\end{itemize} 
\end{conj}

We wanted to show $I=(p_a,p_b)$ is a prime ideal for some $a,b$. We did several computations
on computer and found some conditions on $a,b$. We could not prove these results. We state
them as a conjecture as follows:

\begin{conj}\label{prime computation} Let $S=\mathbb{C}[x_1,x_2,x_3,x_4]$ be the polynomial ring. 
\begin{enumerate}
\item [1.] Let $I=(p_{a},p_{b})$ where $a$ is a prime number, $a \geq 5 $ and $b=a+m+6d 
$ with $m \in \{1,5\} \text{ and } d \in \mathbb{N} \cup \{0\}$. Then $I$ is a prime ideal in $S$. 
Therefore $p_{a},p_{b},p_n$ forms a regular sequence in $S$ for all $p_n \notin (p_{a},p_{b})$.
\item [2.] Let $I=(p_{2},p_{m})$ with $m\in \mathbb{N}$ and $m \neq 2+ 3k, 2+4k$, 
where $k \in \mathbb{N}$. Then $I$ is a prime ideal in $S$. Therefore $p_{2},p_{m},p_n$ forms a 
regular sequence in $S$ for all $p_n \notin (p_{2},p_{m})$.
\item [3.] Let $I=(p_{3},p_{2m})$ with $m\in \mathbb{N}$ and $m \neq 6+ 9 \lambda,\text{ where } 
\lambda \in \mathbb{N}\cup \{0\}.$ Then $I$ is a prime ideal in $S$. Therefore $p_{3},p_{2m},p_n$ 
forms a regular sequence in $S$ for all $p_n \notin (p_{3},p_{2m})$.
\item [4.] Let $I=(p_{4},p_{m})$ with $m\in \mathbb{N}$ and $m \neq 4+ 3k, 4+8k$, where $k \in \mathbb{N}$. 
Then $I$ is a prime ideal in $S$. Therefore $p_{4},p_{m},p_n$ forms a regular sequence in $S$ for
all $p_n \notin (p_{4},p_{m})$.
\end{enumerate} 
\end{conj}
\medskip

Recall the conjecture on power sum symmetric polynomials in the polynomial ring in four variable:
\begin{conj}{\rm (Conca, Krattenthaler, Watanabe)} \\[1mm]
Let $A \subset \mathbb{N}^*$ with $|A|=4$, say $A=\{a_1,a_2,a_3,a_4\}$, 
and assume $\gcd(A)=1.$ Then $p_A(4)$ is a regular sequence if and only if 
$A$ satifies the following conditions:
\begin{enumerate}
 \item [1.] At least two of the $a_i$'s are even, at least one is a multiple of $3$,
            and at least one is a multiple of $4$,
 \item [2.] If $E$ is the set of the even elements in $A$ and $d=\gcd(E)$ then the set
            $\{ \frac{a}{d} : a \in E \}$ contains an even number.
\item [3.] $A$ does not contain a subset of the form $\{d,2d,5d\}$.
\end{enumerate}
 \end{conj}

\medskip

We use the Newton's formula for power sum to deduce the following result:
\begin{pro}
Consider the power sum sequence $p_1,p_2,p_3,p_n$, then
\[
p_n= = \begin{cases}
    (-1)^k4e_4^k \mod (p_1,p_2,p_3),  &\text{ if  $ n=4k $;}\\
    0 \mod (p_1,p_2,p_3),  &\text{ otherwise.}
    \end{cases}
\] 
\end{pro}
\begin{proof}
 Similar to Proposition \ref{p12n}. The only difference in this case is, $p_0=4$.
\end{proof}
\begin{cor}
The power sum sequence $p_1,p_2,p_3,p_n$ is a regular sequence if and only if $n=4k$.
\end{cor}

\begin{proof}
 Similar to Corollary \ref{p12n-cor}. The only difference is, there we use Theorem 
\ref{thm-being-regular} for three variable. Similar result can be deduced for the four variable case also.
\end{proof}

\begin{pro}
Consider the power sum sequence $p_1,p_2,p_4,p_n$, then
\[
p_n= \begin{cases}
    4e_3^k \mod (p_1,p_2,p_4),  &\text{ if  $ n=3k $;}\\
    0 \mod (p_1,p_2,p_4),  &\text{ otherwise.}
    \end{cases}
\] 
\end{pro}
\begin{proof}
Similar to Proposition \ref{p12n}.
\end{proof}
\begin{cor}
The power sum sequence $p_1,p_2,p_4,p_n$ is a regular sequence if and only if $n=3k$.
\end{cor}
\begin{proof}
Similar to Corollary \ref{p12n-cor}.
\end{proof}

\medskip
\subsection { Complete symmetric polynomials in $4$ variables}
\medskip
\begin{pro}
Consider the sequence $h_1,h_2,h_3,h_n$, then
\[
h_n= \begin{cases}
    (-1)^ke_4^k \mod (h_1,h_2,h_3),  &\text{ if  $ n=4k $;}\\
    0 \mod (h_1,h_2,h_3),  &\text{ otherwise.}
    \end{cases}
\]  
\end{pro}
\begin{proof}
Similar to Proposition \ref{h12n}.
\end{proof}

\begin{cor}
The sequence $h_1,h_2,h_3,h_n$ is a regular sequence if and only if $n=4k, k \in \mathbb{N}$. 
\end{cor}
\begin{proof}
 Similar to Corollary \ref{h12nc}. The only difference is one has a similar result
to Theorem \ref{thm-being-regular} in the four variable case.
\end{proof}

\begin{pro}
Consider the sequence $h_1,h_2,h_4,h_n$, then
 \[
h_n= \begin{cases}
    e_3^k \mod (h_1,h_2,h_4),  &\text{ if  $ n=3k $;}\\
    0 \mod (h_1,h_2,h_4),  &\text{ otherwise.}
    \end{cases}
\] 
\end{pro}
\begin{proof}
 Similar to Proposition \ref{h12n}.
\end{proof}

\begin{cor}
The sequence $h_1,h_2,h_4,h_n$ is a regular sequence if and only if $n=3k, k \in \mathbb{N}$.
\end{cor}
\begin{proof}
 Similar to Corollary \ref{h12nc}.
\end{proof}

\begin{pro}
Consider the sequence $h_2,h_3,h_4,h_n$, then
\[
h_n= \begin{cases}
    (-1)^ke_1^{k}e_4^k \mod (h_2,h_3,h_4),  &\text{ if  $ n=5k $;}\\
    (-1)^ke_1^{k+1}e_4^k \mod (h_2,h_3,h_4),  &\text{ if  $ n=5k+1 $;}\\
    0 \mod (h_2,h_2,h_4),  &\text{ otherwise.}
    \end{cases}
\] 
\end{pro}
\begin{proof}
 Similar to Proposition \ref{h12n}.
\end{proof}

\begin{cor}
The sequence $h_2,h_3,h_4,h_n$ is a regular sequence if and only if $n=5k,5k+1, k \in \mathbb{N}$. 
\end{cor}
\begin{proof}
Similar to Corollary \ref{h12nc}. 
\end{proof}
\medskip

\section{Appendix}
\medskip
\subsection{Verification of the Conca, Krattenthaler and Watanabe conjecture}
\begin{itemize}
\item[(i)] For $p_1,p_2,p_n$: This follows directly from Corollary \ref{p12n-cor}.
\item[(ii)] For $p_1,p_3,p_n$: This follows directly from Corollary \ref{p13n-cor}.
\item[(iii)] For $h_1,h_2,h_n$: Let us start with the necessary part.
\begin{enumerate}
 \item [1] $2c \equiv 0 (\mod 6)$ implies that $c=3k$.
 \item [2] $\text{gcd}(2,3)=1$ (always true).
 \item [3] For all $t \in \mathbb{N}$ with $t> 2$ there exist $d \in A$ such that 
$d+2 \not\equiv 0,1 (\mod t)$: we know that $1+2$ and $2+2$ are $0$ or $1$ only 
modulo $3$ and so t=3; these means that this condition is false 
iff $c+2\cong 0, 1 (\mod 3)$. Thus, we have that $c=3k$.
\end{enumerate}
Viceversa, if $n=3k$, $[1]$ and $[3]$ are fulfilled.
\end{itemize}
Similarly we can show, $(iv)$ for $h_1,h_3,h_n\;$; $(v)$ for $h_1,h_4,h_n$, and 
$(vi)$ for $h_2,h_3,h_n$ respectively.
\medskip

\begin{ex}\label{ex-1-4-5}
This example is a case when $h_k\notin (h_i, h_j)$ and $h_i, h_j, h_k$ is 
not a regular sequence. 
\end{ex}
Consider the triple $h_1,h_4,h_5,$ By Proposition \ref{pro-1-4}, we have $h_5=-e_2e_3 \mod(h_1,h_4)$, hence $h_5
\notin (h_1,h_4)$, still $h_1,h_4,h_5$ is not a regular sequence. First we compute 
the hilbert series of $(h_1,h_4,h_5)$ and we find that
\[
 H_{S/(h_1, h_4, h_5)}(t)= \frac{1-t-t^4+t^6+t^7-t^8}{(1-t)^3}.
\]
If $h_1,h_4,h_5$ were a regular sequence, Hilbert series should have been
\[
\begin{split}
    H_{S/(h_1, h_4, h_5)}(t)=&\frac{(1-t)(1-t^4)(1-t^5)}{(1-t)^3}, \\
                           =&\frac{1-t-t^4+t^6+t^9-t^{10}}{(1-t)^3},
\end{split}
\]
which is clearly not the case.
\medskip
\section*{Acknowledgements}
We thanks to Professors Ralf Fr\"oberg, Mats Boij and Alexander Engstr\"om for 
the support and the help given during the Pragmatic 2011. We thanks Professors Aldo Conca
for his valuable suggestions, comments, and writing the CoCoA program for prime ideal test. 
We also thank Professors Anna M. Bigatti and John Abbott for their help in CoCoA for the 
necessary computations.

\end{document}